\documentclass{article}[10]
\usepackage{amsmath,amssymb,amsthm}
\title{Stable determination of a body immersed in a fluid: the nonlinear stationary case.}
\author{ Andrea Ballerini \footnote{SISSA, Via Bonomea 265, 34136 Trieste, Italy. E-mail: balleand@sissa.it.}}
\date{}
\newcommand{\tmop}[1]{\ensuremath{\operatorname{#1}}}
\newcommand{\tmtextbf}[1]{{\bfseries{#1}}}
\newcommand{\tmtextit}[1]{{\itshape{#1}}}

\newtheorem{theorem}{Theorem}[section]
\newtheorem{lemma}[theorem]{Lemma}
\newtheorem{proposition}[theorem]{Proposition}
\newtheorem{corollary}[theorem]{Corollary}
\newenvironment{definition}[1][Definition]{\begin{trivlist}
\item[\hskip \labelsep {\bfseries #1}]}{\end{trivlist}}
\newenvironment{remark}[1][Remark]{\begin{trivlist} \theoremstyle{remark}
\item[\hskip \labelsep {\bfseries #1}]}{\end{trivlist}}

\newcommand{\dive}{\mathrm{div} {\hspace{0.25em}}}
\newcommand{\en}{\mathcal{E}}
\newcommand{\elledue}[1]{{\bf L}^2({#1})}
\newcommand{\accauno}[1]{{\bf H}^1({#1})}
\newcommand{\accan}[2]{{\bf H}^{#1}({#2})}
\newcommand{\ide}{{\hspace{0.25em}}\mathbb{I}}
\newcommand{\til}[1]{\widetilde{#1}}
\newcommand{\omegad}{{\Omega} {\setminus} \overline{D}}
\newcommand{\norma}[3]{\|#1\|_{\accan{#2}{#3}}}
\newcommand{\normadue}[2]{\|#1\|_{\elledue{#2}}}
\newcommand{\accano}[2]{{\bf H}^{#1}_0({#2})}

\newcommand{\gio}{ \| g \|_{{\bf C}^{1,\alpha} (\Gamma) }}
\newcommand{\psii}{\| \psi \|_{{\bf H}^{-\frac{1}{2}} (\Gamma)}}

\numberwithin{equation}{section}
\begin{document}

\maketitle
\begin{abstract} 
We consider the inverse problem of the detection of a single body, immersed in a bounded container filled with a fluid which obeys the stationary Navier-Stokes equations, from a single measurement of force and velocity on a portion of the boundary. We obtain an estimate of stability of log-log type.
\end{abstract}

\section{Introduction.}
In this paper we deal with an inverse problem associated to the stationary Navier-Stokes system. We consider a bounded set $\Omega \subset \mathbb{R}^n$ (we assume $n=2,3$, which are indeed the physically relevant cases) with a sufficiently smooth boundary $\partial \Omega$. This body is filled with a Navier-Stokes fluid of constant viscosity $\mu$. We want to detect an object $D$ immersed in this container, by collecting measurements of the velocity of the fluid motion and of the boundary forces, but we only have access to a portion $\Gamma$ of the boundary $\partial \Omega$. Once a suitable boundary condition $g$ is assigned on $\Gamma$, the velocity $u=(u_1, \dots, u_n)$ and the pressure $p$ of the fluid will obey the following Navier-Stokes system in $\omegad$:
\begin{equation}
  \label{NSE} \left\{ \begin{array}{rl}
    \dive\sigma(u,p) &= (u \cdot \nabla) u  \hspace{2em} \mathrm{\tmop{in}} \hspace{1em}
    \omegad,\\
    \dive u & = 0 \, \, \, \hspace{2em} \mathrm{\tmop{in}} \hspace{1em} \omegad,\\
    u & = g \hspace{2em}\, \, \,  \mathrm{\tmop{on}} \hspace{1em} \Gamma,\\
    u & = 0 \hspace{2em} \, \, \, \mathrm{\tmop{on}} \hspace{1em} \partial D.
  \end{array} \right.
\end{equation}
Here, \begin{displaymath} \sigma (u, p) =  \mu ( \nabla u + \nabla u ^T )  -  p \ide   \end{displaymath} is the \tmtextit{stress tensor}, where $\ide$ denotes the $n \times n$ identity matrix. The $i$th component of the nonlinear term $(u \cdot \nabla) u$ is given by
\begin{equation} \label{nonlin}
\big( (u \cdot \nabla) u\big)_i = \sum_{j=1}^{n} u_j \frac{\partial u_i}{\partial x_j}.
\end{equation}
 The last equation in \eqref{NSE} is  called ``no-slip condition''.  
Call $\nu$ the outer normal vector field to $\partial \Omega$.
Once  $g \in \accan{\frac{1}{2}}{\Gamma}$ is assigned, one can, in principle, measure on $\Gamma$ the induced normal component of the stress tensor
\begin{equation} \psi = \label{psi}\sigma (u, p) \cdot \nu, \end{equation} 
and try to recover $D$ from a single pair of Cauchy data $(g, \psi)$ known on the accessible part of the boundary $\Gamma$. 
Under some additional regularity hypotheses, namely of $\partial \Omega$ being of Lipschitz class, and  $g \in \accan{\frac{3}{2}}{\Gamma}$,  the uniqueness for this inverse problem has been shown to hold (see \cite{ConcOrtega2}) by means of unique continuation techniques. This means that if $u_1$ and $u_2$ are two solutions of \eqref{NSE} corresponding to an assigned boundary data $g$, for $D=D_1$ and $D=D_2$ respectively, and $\sigma(u_1, p_1) \cdot \nu= \sigma(u_2, p_2) \cdot \nu$ on $\Gamma$, then $D_1=D_2$.
This paper is devoted to the analysis of the problem of stability, which we may roughly state as follows: \\
\emph{Given two solutions $(u_i,p_i)$ to \eqref{NSE}  for two different  $D_i$, for $i=1,2$, with the same boundary data $g$, if }
\begin{displaymath}
\|\sigma(u_1, p_1) \cdot \nu - \sigma(u_2, p_2) \cdot \nu  \| \le \epsilon,
\end{displaymath}
\emph{ what is the rate of convergence of  $\mathrm{d}_{\mathcal{H}}(D_1, D_2) $ as $\epsilon \to 0$?} \\  (We denote by $\mathrm{d}_{\mathcal{H}}$ the Hausdorff distance).
In \cite{MEE} we proved a rate of convergence of log-log type for the analogous problem in the simpler context of the Stokes system. Here, we will prove an equivalent result for  the stationary Navier-Stokes equations. 
As for the (yet easier) Stokes problem, even if we add some a priori information on the regularity of the unknown domain, we can only obtain a weak rate of stability. This does not come unexpected since, even for much simpler problems of the same kind, the dependence of $D$ from the Cauchy data is at most of logarithmic type. See, for example, \cite{ABRV} for a similar problem on electric conductivity, or \cite{MRC}, \cite{MR} for an inverse problem regarding elasticity. There are, in fact, several counterexamples showing that the optimal rate of convergence for the inverse conductivity problem is no better than of log type (see \cite{Aless1} and \cite{DiCriRo}). \\ 
The purpose of this paper is thus to prove a log-log type stability for the Hausdorff distance between the boundaries of the inclusions, assuming a $C^{2,\alpha}$ regularity bound. Such estimates have been estabilished for various kinds of elliptic equations, for example, \cite{ABRV}, \cite{AlRon}, for the electric conductivity equation, \cite{MRC}, \cite{MR} for the elasticity system and the detection of cavities or rigid inclusions, and \cite{MEE} for the Stokes equation.
The main tool used to prove stability here and in the aforementioned papers \cite{ABRV}, \cite{MRC}, \cite{MR}, \cite{MEE} is a quantitative estimate of continuation from boundary data, in the interior, in the form of a three spheres inequality (see Theorem \ref{teotresfere}) and its main consequences. However, while in \cite{ABRV} the estimates are of log type for a scalar equation, here, and in \cite{MRC}, \cite{MR} and \cite{MEE}, only an estimate of log-log type could be obtained for a system of equations. To improve that, one would need a doubling inequality at the boundary for systems of equations, which basically would allow to extend the reach of the unique continuation property up to the boundary. Unfortunately, to the present time, none are available; on the other hand they are known to hold in the scalar case.\\ 
A very recent paper by Lin, Uhlmann and Wang (\cite{ULW}) extended the validity of the three spheres inequality to linearized Navier-Stokes systems: this allows us to apply it to differences of solutions of \eqref{NSE}, see Proposition \ref{tredifferenze}. 
In order to adapt this result to the Navier-Stokes equations, however, we are forced to add yet more a priori hypotheses on the solutions, mainly because of the nonlinear character of the equations. Proposition \ref{tredifferenze}, in fact, applies to linearized stationary Navier-Stokes systems with coefficients bounded in an appropriate norm. We meet this request by restricting the choice of boundary data, i.e. we require a strong regularity bound- i.e., $\mathbf{C}^{1,\alpha}$-on the boundary data. \\
The present paper has the same structure as \cite{MEE} and closely follows \cite{MRC}, \cite{MR}.  
The main ingredients are the following: 
\begin{enumerate} \item {\it An estimate of propagation of smallness from the interior}. The proof of this estimate relies essentially on the three spheres inequality. Once we adapt it to \eqref{NSE}, the estimate itself is adapted effortlessly. 
  \item {\it A stability estimate of continuation from the Cauchy data}. This result also relies heavily on the three spheres inequality, but in order to obtain a useful estimate of continuation near the boundary, we need to extend a given solution of \eqref{NSE} to a larger domain, so that the extended solution solves a similar system of equations. We may then apply the stability estimates from the interior to the extended solution and treat them like estimates near the boundary for the original one. 
\end{enumerate} 
In Section 2, we state the apriori hypotheses we will need throughout the paper, recall some useful properties of the direct problem, and give the main result in Theorem \ref{principale}. Section 3 contains the estimates of continuation in Propositions \ref{teoPOS} and \ref{teoPOSC}, and Propositions \ref{teostabest} and \ref{teostabestimpr}. These deal, in turn, with the  stability estimates of continuation from Cauchy data and an improved version of the latter under some additional regularity hypotheses. We then use them for the proof of Theorem \ref{principale}. In section 4, we derive several versions of the three spheres inequality (Theorems \ref{teotresfere}, \ref{teotresferegrad} and Proposition \ref{tredifferenze}), which are needed to prove Proposition \ref{teoPOS}. 
Section 5 is devoted to the proof of Proposition \ref{teostabest}, which will use an estimate of continuation from Cauchy data, Theorem \ref{stabilitycauchy}, which will be proven in Section 6. 
\section{The stability result.}
\subsection{Notation and definitions.}

Let $x\in \mathbb{R}^n$. We will denote by $ B_{\rho}(x)$ the ball in $\mathbb{R}^n$ centered in $x$ of radius $\rho$. We will indicate  $x = (x_1, \dots ,x_n) $ as $x= (x^\prime, x_n)$ where $x^\prime = (x_1 \dots x_{n-1})$. Accordingly, $B^\prime_{ \rho}(x^\prime)$ will denote the ball of center $x^\prime$ and radius $\rho$ in $\mathbb{R}^{n-1}$. 
We will often make use of the following definition of regularity of a domain.
\begin{definition}
Let $\Omega \subset \mathbb{R}^n$ a bounded domain. We say $\Gamma \subset \partial \Omega$  is of class $C^{k, \alpha}$  with constants $\rho_0$, $M_0 >0$, where $k$ is a nonnegative integer, $\alpha \in [ 0,1 )$ if, for any $P \in \Gamma$ there exists a rigid transformation of coordinates in which $P = 0$ and  
\begin{displaymath}
\Omega \cap B_{\rho_0}(0) = \{ (x^\prime, x_n) \in  B_{\rho_0}(0) \, \, \mathrm{s.t. } \, \, x_n > \varphi (x^\prime)\},
\end{displaymath}
where $\varphi$ is a real valued function of class $C^{k, \alpha}(B^\prime_{\rho_0}(0))$ such that \begin{displaymath} \begin{split}  \varphi(0)&=0, \\ \nabla\varphi(0)&=0, \text{  if  } k \ge 1 \\ \| \varphi\|_{C^{k, \alpha}(B^\prime_{\rho_0}(0))} &\le M_0 \rho_0. 
\end{split}
\end{displaymath}
\end{definition}
When $k=0$, $\alpha=1$ we will say that $\Gamma$ is {\it of Lipschitz class with constants $\rho_0$, $M_0$}.

\begin{remark} We normalize all norms in such a way they are all dimensionally equivalent to their argument and coincide with the usual norms when $\rho_0=1$. In this setup, the norm taken in the previous definition is intended as follows:
\begin{displaymath}
\| \varphi\|_{C^{k, \alpha}(B^\prime_{\rho_0}(0))} = \sum_{i=0}^{k} \rho_0^i \| D^i \varphi\|_{L^{\infty}(B^\prime_{\rho_0}(0))} + \rho_0^{k+\alpha}  | D^k \varphi |_{\alpha,B^\prime_{\rho_0}(0) },
\end{displaymath}
where $| \cdot |$ represents the $\alpha$-H\"older seminorm 

\begin{displaymath}
| D^k \varphi |_{\alpha,B^\prime_{\rho_0}(0) } = \sup_{x^\prime, y^\prime \in B^\prime_{\rho_0}(0), x^\prime \neq y^\prime }  \frac{| D^k \varphi(x^\prime)-D^k \varphi(y^\prime)| }{|x^\prime -y^\prime|^\alpha},
\end{displaymath}
and $D^k \varphi=\{ D^\beta\varphi\}_{|\beta|= k}$ is the set of derivatives of order $k$.
Similarly we set 
\begin{displaymath}
\normadue{u}{\Omega}^2 = \frac{1}{\rho_0^n}   \int_\Omega u^2 \,
\end{displaymath}

\begin{displaymath}
\norma{u}{1}{\Omega}^2 = \frac{1}{\rho_0^n} \Big( \int_\Omega u^2 +\rho_0^2 \int_\Omega |\nabla u|^2 \Big).
\end{displaymath}
The same goes for the trace norms $\norma{u}{\frac{1}{2}}{\partial \Omega}$ and the dual norms $\norma{u}{-1}{\Omega}$, $\norma{u}{-\frac{1}{2}}{\partial \Omega}$ and so forth. \end{remark}
\subsection{A priori information.}
Here we present all the a priori hypotheses we will use along the paper. \\
(1) {\it A priori information on the domain.}  \\
We assume $\Omega \subset \mathbb{R}^n$ to be a bounded domain, such that 
\begin{equation} \label{apriori0}
 \partial \Omega \text{  is connected,  } 
\end{equation}
with a sufficiently smooth boundary, i.e., 
\begin{equation} \label{apriori1}
\partial \Omega \text{ is of class } C^{2, \alpha} \text{ of constants } \rho_0, \, \, M_0, \end{equation} where $\alpha \in (0,1]$ is a real number, $M_0 > 0$, and $\rho_0 >0 $ is what we shall treat as our dimensional parameter. In what follows $\nu$ is the outer normal vector field to  $\partial \Omega$. We also require that 
\begin{equation} \label{apriori2} |\Omega| \le M_1 \rho_0^n, \end{equation} where $M_1 > 0$. \\
We choose an open and connected portion $\Gamma \subset \partial \Omega$ as being the accessible part of the boundary. We assume that there exists a point $P_0 \in \Gamma$ such that 
\begin{equation} \label{apriori2G}
\partial \Omega \cap B_{\rho_0}(P_0) \subset \Gamma.
\end{equation} 
(2) { \it A priori information about the obstacles.} \\
We consider $D \subset \Omega$, which represents the obstacle we want to detect from the boundary measurements, on which we require that 
\begin{equation} \label{apriori2bis}
\Omega \setminus \overline{D}  \text{ is connected, } 
\end{equation}
\begin{equation} \label{apriori2ter}
\partial D \text{ is connected. }
\end{equation}
We require the same regularity on $D$ as we did for $\Omega$, that is, 
\begin{equation} \label{apriori3} \partial D \text{ is of class } C^{2, \alpha} \text{ with constants }    \rho_0 , \, M_0.  \end{equation} 
In addition, we suppose that  \begin{equation} \label{apriori4} d (D, \partial \Omega) \ge \rho_0. \end{equation}
(3) { \it A priori information about the boundary data.} \\
For the Dirichlet-type data $g$ we assign on the accessible portion of the boundary $\Gamma$, we assume that
\begin{equation} \begin{split} \label{apriori5}
g \in {\bf C}^{1,\alpha}(\partial \Omega), \, \; \; g \not \equiv 0,  \\ \mathrm{supp} \,g \subset \subset \Gamma.
\end{split} \end{equation} 
We  prescribe the following compatibility condition (which is  necessary for the existence of the solution, and is a consequence of the incompressibility condition):
\begin{equation} \label{aprioriexist} \int_{\partial \Omega}  g \, \mathrm{d} s =0.
\end{equation}
We shall also assume an apriori bound on the regularity of the flow, by requiring  that for a given constant $\mathcal{E}$ we have
\begin{equation} \label{aprioriE}
\gio \le \mathcal{E}.
\end{equation}
We also specify a bound on the oscillation of the boundary data $g$ by requiring that, for a given constant $F>0$, 
\begin{equation} \label{apriori7}
\normadue{g}{\Gamma}  \ge F.
\end{equation}
Note that \eqref{aprioriE} and \eqref{apriori7} combined yield an a priori frequency type limitation of the form \begin{displaymath}\frac{\gio  }{\normadue{g}{\Gamma}}  \le \frac{\mathcal{E}}{F}. \end{displaymath}
Under the above conditions on $g$, one can prove that there exists a constant $c>0$, only depending on $M_0$, such that the following equivalence relation holds: 
\begin{equation} \label{equivalence} 
\frac{1}{c}\gio \le  \norma{g}{\frac{1}{2}}{\partial \Omega} \le  c \gio.
\end{equation}
\subsection{The direct problem}
In this section we collect some useful results regarding the direct problem of finding weak solutions of \eqref{NSE}. We begin by an existence result (which is a classical result and can be found in \cite{Galdi}, \cite{Temam}) :
\begin{theorem} \label{TeoRegGen} 
Let  $g \in \mathbf{C}^{1,\alpha} (\partial \Omega)$ satisfy \eqref{apriori5} and \eqref{aprioriexist}, let $\Omega \subset \mathbb{R}^n$ a bounded set satisfying \eqref{apriori0}-\eqref{apriori2G}, and let $D \subset \Omega$ satisfy \eqref{apriori2bis}-\eqref{apriori4}. Then for every $\mu >0$ there exists at least one solution $(u,p) \in \accauno{\Omega \setminus D} \times \elledue{\Omega \setminus D}$  of \eqref{NSE}, and 
\begin{equation} \label{stimanormadiretto}
\|u \|_{\mathbf{H}^1 (\Omega \setminus D ) } \le C \gio,
\end{equation}
with $C$ only depending on $\mu$, $M_0$, $M_1$. 
\end{theorem}
The following is a consequence of the above theorem combined with a global regularity result for solutions of the Navier-Stokes equations, see, for example, \cite{Galdi}[VIII, Corollary 5.2]. We point out that an a priori higher degree of integrability of weak solutions has to be assumed.
\begin{theorem}[Regularity of solutions] \label{teoschauder} Assume that the hypotheses of the previous theorem are satisfied and suppose, in addition, that \eqref{aprioriE} is satisfied. Let  $u$ be the weak solution to \eqref{NSE} in $\Omega \setminus D$, and suppose that $u \in \accauno{\Omega \setminus D} \cap \mathbf{L}^n (\Omega \setminus D)$. Then for all $0<\alpha<1$ we have that 
$u \in C^{1,\alpha}(\overline{\Omega \setminus D})$ and 
\begin{equation} \label{schauder1} 
\|u \|_{C^{1,\alpha}(\overline{ \Omega \setminus D})} \le C \mathcal{E},
\end{equation} 
where $C>0$ only depends on $\mu$, $\alpha$, $M_0$.  \end{theorem}
\begin{remark}
The requirement that  $u \in \accauno{\Omega \setminus D} \cap \mathbf{L}^n (\Omega \setminus D)$ is actually redundant (in the sense that it follows from  $u \in \accauno{\Omega \setminus D}$ ) when $n \le 4$, due to the Sobolev embedding theorems. 
\end{remark}
The uniqueness issue for the direct problem is more subtle. Unlike the linear Stokes equations, here uniqueness is not guaranteed in general. In fact, several examples can be built even in low space dimensions (see \cite{Temam}). As far as the inverse problem is concerned, whether or not the solution of the direct problem is unique is not relevant, for by formulating the inverse problem we implicitly select one particular solution of the direct problem to work with. 
To guarantee uniqueness, one can either a priori bound the norm of the solution (see \cite{Galdi}, Theorem VIII.2.1), or take ``not too large'' boundary data $g$ and viscosity $\mu$, as stated by the following  (which we will state, for simplicity, only for $n \le 4$; see \cite{Temam}, Theorem 1.6, pg. 120 for a proof):
\begin{theorem}[Uniqueness for small data] \label{unique}
Let $\Omega$ and $g$ satisfy the same hypotheses of Theorem \ref{TeoRegGen}, and let $u$ be a solution of \eqref{NSE} given by Theorem \ref{TeoRegGen}. Then:
\begin{enumerate}
\item There exists a constant $C_1= C_1(\mu, \Omega)$ such that if $w$ is another solution of \eqref{NSE} and $\|u \|_{\mathbf{H}^1(\Omega \setminus D)} \le C_1$ then $w=u$.
\item There exists a constant $C_2= C_2(\mu, \Omega, n)$ such that, if $ \gio \le C_2$, then $u$ is the unique solution of \eqref{NSE}.
\item There exists a constant $C_3= C_3(\Omega, n, \gio)$ such that if $\mu \ge C_3$ then $u$ is the unique solution of \eqref{NSE}.
\end{enumerate}
\end{theorem}
\subsection{The main result.}
 Let $\Omega \subset \mathbb{R}^n$, and $\Gamma \subset \partial \Omega$ satisfy \eqref{apriori1}-\eqref{apriori2G}. Let $D_i \subset \Omega$, for $i=1,2$, satisfy \eqref{apriori2bis}-\eqref{apriori4}, and let us denote by $\Omega_i= \Omega \setminus \overline{D_i}$. We may state the main result as follows.
\begin{theorem}[Stability] \label{principale} Let $g \in \mathbf{C}^{1,\alpha} (\Gamma)$ be the assigned boundary data, satisfying \eqref{apriori5}-\eqref{apriori7}. Let $u_i \in \accauno{\Omega_i} \cap \mathbf{L}^n (\Omega_i)$ solve \eqref{NSE} for $D=D_i$. If, for $\epsilon > 0 $, we have 
\begin{equation} \label{HpPiccolo}
\rho_0 \norma{\sigma(u_1, p_1)\cdot \nu  -\sigma(u_2,p_2) \cdot \nu }{-\frac{1}{2}}{\Gamma} \le \epsilon, \end{equation}
then 
\begin{equation}\label{stimstab} 
d_{\mathcal{H}} (\partial D_1, \partial D_2) \le \rho_0 \omega(\epsilon),
\end{equation}
where $\omega : (0, +\infty) \to \mathbb{R}^+$ is an increasing function satisfying, for all $0<t<\frac{1}{e}$:
\begin{equation}
\omega(t) \le C (\log | \log t |)^{-\beta   }.
\end{equation}
The constants $C>0$ and $0<\beta<1$ only depend on $\mu$, $n$, $M_0$, $M_1$, $\en$ and $F$.
\end{theorem}
\section{Proof of Theorem \ref{principale}.}
In what follows, we shall repeatedly use the following notation, for various choices of the open set $\Omega$ and the parameter $h>0$:
\begin{displaymath}
\Omega_h  = \{ x \in \Omega \, \, \mathrm{such \, \, that } \, \, d(x, \partial \Omega) > h  \}.
\end{displaymath}
The proof of Theorem \ref{principale} relies on the following sequence of propositions.
\begin{proposition}[Lipschitz propagation of smallness] 
\label{teoPOS}
Let $E$ be a bounded Lipschitz domain with constants $\rho_0$, $M_0$, satisfying \eqref{apriori2}. 
Let $u$ be a solution to the following problem:
\begin{equation}
  \label{NSEPOS} \left\{ \begin{array}{rl}
    \dive\sigma(u,p) &= (u \cdot \nabla) u \hspace{2em} \mathrm{\tmop{in}} \hspace{1em}
    E,\\
    \dive u & = 0 \hspace{2em} \mathrm{\tmop{in}} \hspace{1em} E,\\
    u & = g \hspace{2em} \mathrm{\tmop{on}} \hspace{1em} \partial E,\\
  \end{array} \right.
\end{equation}
where $g$ satisfies 
\begin{equation} \label{apriori5POS}
g \in \mathbf{C}^{1,\alpha}(\partial E),
\, \; \; g \not \equiv 0, 
\end{equation} 
\begin{equation} \label{aprioriexistPOS} \int_{\partial E} g \,  \mathrm{d} s =0,
\end{equation}
\begin{equation} \label{aprioriEPOS} \| g \|_{\mathbf{C}^{1,\alpha} (\partial E) } \le \en,
\end{equation}
\begin{equation} \label{apriori7POS}
\normadue{g}{\partial E}  \ge F,
\end{equation}
for  given constants $\en>0$, $F>0$. Also suppose that there exists a point $P \in \partial E$ such that
\begin{equation} 
g = 0  \;\; \text{on}  \; \; \partial E \cap B_{\rho_0}(P).
\end{equation} Then there exists a constant $s>1$, depending only on $n$ and $M_0$ such that, for every $\rho >0$ and for every $\bar{x} \in E_{s\rho}$, we have 
\begin{equation} \label{POS}  \int_{B_{\rho}(\bar{x})} \! |\nabla u|^2 dx \ge C_\rho \int_{E} \! |\nabla u|^2 dx . \end{equation} 
Here $C_\rho>0$ is a constant depending only on  $n$, $M_0$, $M_1$, $F$, $\en$, $\rho_0$ and $\rho$. 
The constant $C_\rho$ can be written as
\begin{equation} \label{crho} C_\rho = \frac{C}{\exp \Big[ A \big( \frac{\rho_0}{\rho}\big) ^B \Big] } \end{equation} where $A$, $B$, $C>0$ only depend on $n$, $M_0$, $M_1$, $F$ and $\en$.
\end{proposition}
\begin{proposition}[Lipschitz propagation of smallness up to boundary data] \label{teoPOSC}
Under the hypotheses of Theorem \ref{principale}, for all $\rho>0$, if $\bar{x} \in (\Omega_i)_{{{(s+1)\rho}}}$,  we have for $i=1,2$:
\begin{equation} \label{POScauchy}
\frac{1}{\rho_0^{n-2}} \int_{B_{\rho}(\bar{x})} \! |\nabla u_i|^2 dx \ge  C_\rho \gio^2,
\end{equation}
where $C_\rho$ is as in \eqref{crho} (with possibly a different value of the term $C$), and $s$ is given as in Proposition \ref{teoPOS}.
\end{proposition} 
\begin{proposition}[Stability estimate of continuation from Cauchy data] 
\label{teostabest} Under the hypotheses of Theorem \ref{principale} we have 
\begin{equation}\label{stabsti1} \frac{1}{\rho_0^{n-2}} \int_{D_2\setminus D_1} |\nabla u_1|^2 \le C \omega(\epsilon) \end{equation}
\begin{equation}\label{stabsti2} \frac{1}{\rho_0^{n-2}} \int_{D_1\setminus D_2} |\nabla u_2|^2 \le C \omega(\epsilon) \end{equation}
where $\omega$ is an increasing continuous function, defined on $\mathbb{R}^+$ and satisfying 
\begin{equation}\label{andomega} \omega(t) \le C \big( \log |\log t|\big)^{-c} \end{equation} 
for all $t < e^{-1}$, where $C$ only depends on $\mu$, $n$, $M_0$, $M_1$ and $\en$, and $c>0$ only depends on $n$. 
\end{proposition}
\begin{proposition}[Improved stability estimate of continuation] \label{teostabestimpr}
Let the hypotheses of Theorem \ref{principale} hold. Let $G$ be the connected component of $\Omega_1 \cap \Omega_2$ containing $\Gamma$, and assume that $\partial G$ is of Lipschitz class of constants $\tilde{\rho}_0$ and $\tilde{M_0}$, where $M_0>0$ and $0<\tilde{\rho}_0<\rho_0$. Then \eqref{stabsti1} and \eqref{stabsti2} both hold with $\omega$ given by
\begin{equation}\label{omegabetter}
\omega(t)= C |\log t|^{-\gamma},
\end{equation}
defined for $t<1$, where $\gamma >0$ and $C>0$ only depend on $\mu$, $M_0$, $\tilde{M_0}$, $M_1$, $\en$ and $\frac{\rho_0}{\tilde{\rho}_0}$. 
\end{proposition}
\begin{proposition} \label{teoreggra} Let $\Omega_1$ and $\Omega_2$ two bounded domains satisfying \eqref{apriori1}. Then there exist two positive numbers $d_0$, $\tilde{\rho}_0$, with $\tilde{\rho}_0 \le \rho_0$, such that the ratios $\frac{\rho_0}{\tilde{\rho}_0}$, $\frac{d_0}{\rho_0}$ only depend on $n$, $M_0$ 
and $\alpha$ such that, if 
\begin{equation} \label{relgr1}
d_{\mathcal{H}} (\overline{\Omega_1}, \overline{\Omega_2}) \le d_0,
\end{equation}
then there exists $\tilde{M}_0>0$ only depending on $n$, $M_0$ and $\alpha$ such that  every connected component of $\Omega_1 \cap \Omega_2$ has boundary of Lipschitz class with constants $\tilde{\rho}_0$, $\tilde{M}_0$. 
\end{proposition}
We will give a sketch of the proofs of Propositions \ref{teoPOS} and \ref{teoPOSC}  in Section 4, while Propositions \ref{teostabest} and \ref{teostabestimpr} will be proven in Section 5. The proof of Proposition \ref{teoreggra} is purely geometrical and can be found in \cite{ABRV}. 
\begin{proof}[Proof of Theorem \ref{principale}.]
Let us call 
\begin{equation} \label{distanza} d= d_\mathcal{H}(\partial D_1, \partial D_2). \end{equation}
Let $\eta$ be the quantity on the right hand side of \eqref{stabsti1} and \eqref{stabsti2}, so that
\begin{equation} \label{eta} \begin{split}
\int_{D_2 \setminus D_1} |\nabla u_1|^2 \le \eta, \\
\int_{D_1 \setminus D_2} |\nabla u_2|^2 \le \eta. \\
\end{split}\end{equation}
Without loss of generality, assume that there exists a point $x_1 \in \partial D_1$ such that $\mathrm{dist}(x_1, \partial D_2)=d$. We distinguish two possible situations: \\ (i) $B_d(x_1) \subset D_2$, \\ (ii) $B_d(x_1) \cap D_2 =\emptyset$.\\
In case (i), by the regularity assumptions on $\partial D_1$, we find a point $x_2 \in D_2 \setminus D_1$ such that $B_{td}(x_2) \subset D_2 \setminus D_1$, where $t$ is small enough (for example, $t=\frac{1}{1+\sqrt{1+M_0^2}}$ suffices). Using \eqref{POScauchy}, with $\rho = \frac{t d}{s}$ we have 
\begin{equation} \label{stimapos} \int_{B_\rho (x_2) } |\nabla u_1|^2 dx \ge   \frac{C\rho_0^{n-2}}{\exp \Big[A \big(\frac{s\rho_0}{t d }\big)^B\Big]}  \gio^2.  
\end{equation}
By Proposition \ref{teostabest}, we have: 
\begin{equation} \label{quellaconomega}
\omega( \epsilon) \ge \frac{C \gio ^2}{ \exp \Big[A \big(\frac{s\rho_0}{t d }\big)^B\Big]}  , \end{equation}  
which, once we recall \eqref{apriori7} and solve for $d$, yields this estimate of log-log-log type stability:
\begin{equation}\label{logloglog}
d \le C \rho_0 \big\{ \log \big[ \log \big|\log\epsilon \big| \big] \big\}^{-\frac{1}{B}},
\end{equation} 
provided $\epsilon < e^{-e} $: this is not restrictive since, for larger values of $\epsilon$, the thesis is trivial. If we call $d_0$ the right hand side of \eqref{logloglog}, we have that there exists $\epsilon_0$ only depending on $n$, $M_0$, $M_1$ and $F$ such that, if $\epsilon \le \epsilon_0$ then $d\le d_0$. Proposition \ref{teoreggra} then applies, so that $G$ satisfies the hypotheses of Proposition \ref{teostabestimpr}. This means that we may choose  $\omega$ of the form \eqref{omegabetter} in \eqref{quellaconomega}, obtaining \eqref{stabsti1}.
Case (ii) can be treated analogously, upon substituting $u_1$ with $u_2$.
\end{proof}
\section{Proof of Proposition \ref{teoPOS}.}
In a recent paper by Lin, Uhlmann and Wang (\cite{ULW}), the validity of the three spheres inequality has been extended to solutions $u=(u_1, \dots, u_n)$ of Stokes systems of the form
\begin{equation} \label{ellgeneral} 
 \left\{ \begin{array}{rl}
   \triangle u  + A(x) \cdot \nabla u + B(x) u + \nabla p =0 ,
\\
\dive u = 0 ,
   \end{array} \right.
\end{equation}
%
where $A(x)$ is a measurable vector, $B(x)$ is a measurable matrix, both satisfying appropriate bounds, and it is agreed that $A \cdot \nabla u = (A \cdot \nabla u_1, \dots, A \cdot \nabla u_n)$ (the $\cdot$  on the right hand side being the ordinary vector dot product). In what follows, it shall be convenient to write out the first equation in \eqref{ellgeneral} component-wise (here and everywhere else in the paper, we use the convention of summation over repeated indexes): 
\begin{equation} \label{ellgencomp}
 \triangle u_i  + a_{j} \frac{\partial u_i }{\partial x_j} + b_{ij} u_j +\frac{\partial p }{\partial x_i} =0 , \; \mathrm{  for } \; i=1, \dots, n.
\end{equation}
 An analogous result when $B=0$ was obtained by Fabre and Lebeau \cite{FL} and Regbaoui \cite{Regb}. The result in its most general form allows $A$ and $B$ to have some type of singularity; we shall only state a simplified version: we will assume that 
\begin{equation} \label{ablim}
\|A \|_{C^{1,\alpha}(B_R)} +\| B\|_{C^{0,\alpha}(B_R)}    \le \en.
\end{equation}
Then we have:
\begin{theorem}[Three spheres inequality.] \label{teotresfere}
Let  $u \in \accauno{B_R}$  be a solution to \eqref{ellgeneral} in a ball $B_R$. Suppose that the functions $A(x)$, $B(x)$ are measurable and such that \eqref{ablim} holds. 
Then there exists a real number $\vartheta^* \in (0, e^{-1/2})$, depending only on $n$, such that, for all $0<r_1 <r_2 <\vartheta^* r_3$ with $r_3  \le R$ we have:
\begin{equation} \label{tresfere}  \int_{B_{r_2}} \! | u|^2 \le C \Big(\int_{B_{r_1} } \! | u|^2  \Big)^\delta  \Big(\int_{B_{r_3}} \! | u|^2  \Big)^{1-\delta} \end{equation} 
where the balls $B_{r_i}$ are concentric with $B_R$ 
and $\delta \in (0,1)$ and $C>0$ are constants depending only on $\en$, $n$, $\frac{r_1}{r_3}$ and $\frac{r_2}{r_3}$.
\end{theorem}
We will also need to formulate a three spheres type inequality for the first derivatives of $u$:
\begin{corollary} \label{teotresferegrad}
Let $u \in \accauno{B_R}$ be a solution to \eqref{ellgeneral}, and suppose that \eqref{ablim} holds. Assume furthermore that $B(x)\equiv 0$. Then we have that for all $0<r_1 <r_2 <\vartheta^* r_3$ with $r_3  \le R$:
\begin{equation} \label{tresferegrad}  \int_{B_{r_2}} \! |\nabla u|^2 \le C \Big(\int_{B_{r_1} } \! |\nabla u|^2 \Big)^\delta  \Big(\int_{B_{r_3}} \! |\nabla u|^2 \Big)^{1-\delta} \end{equation} 
where $\vartheta^*$ is the same as in Theorem \ref{teotresfere}, $\delta \in (0,1)$ and $C>0$ are constants depending only on $\en$, $n$, $\frac{r_1}{r_3}$ and $\frac{r_2}{r_3}$, and the balls $B_{r_i}$ are concentric with $B_R$.
\end{corollary}
\begin{remark}
We point out that the first equation \eqref{NSE} may be written in the form \eqref{ellgencomp} by writing the nonlinear term as in \eqref{nonlin} and calling $(A(x))_{j}= u_j$ and $(B(x))_{ij}=0$. Therefore, the three spheres inequalities \eqref{tresfere} and \eqref{tresferegrad} may be applied to solutions of \eqref{NSE} in a domain $E$ as long as \eqref{ablim} holds in $E$ with the aforementioned choices of $A$ and $B$. 
\end{remark}
Finally, we would like to recall the following propositions. The first is a  version of Poincar\`e inequality dealing with functions that vanish on an open portion of the boundary (see \cite{Meyers}, or \cite{AleMorRos} for a precise evaluation of the constants in terms of the Poincar\'e constant of the domain and the measure of the portion of the boundary of the domain where the function vanishes on). The second is a Caccioppoli-type inequality for \eqref{ellgeneral}, which can be found in \cite{GiaMod}.
\begin{proposition}[Poincar\`e inequality]
Let $E\subset \mathbb{R}^n$ be a bounded domain with boundary of Lipschitz class with constants $\rho_0$, $M_0$ and satisfying \eqref{apriori2}. Then for every $ u \in  \accan{1}{E}$ such that
\begin{displaymath} 
u = 0  \, \, \text{on}  \, \, \partial E \cap B_{\rho_0}(P), 
\end{displaymath}
where $P$ is some point in $\partial E$, we have
\begin{equation} \label{pancarre}
\normadue{u}{E} \le C \rho_0 \normadue{\nabla u}{E},
\end{equation}
where C is a positive constant only depending on $M_0$ and $M_1$.
\end{proposition}
\begin{proposition}[Caccioppoli inequality]
Let $u \in \accauno{B_R}$ be a solution of \eqref{ellgeneral} in $B_{R}$, and suppose that \eqref{ablim} holds. Then there exists $C>0$ depending only on $n$,$\mu$ and $\mathcal{E}_1$ such that, for every $r$ with $0<r<R$ we have 
\begin{equation} \label{cacio} 
\int_{B_r} |\nabla u|^2 \le \frac{C}{(R-r)^2} \int_{B_R} |u|^2.
\end{equation} 
\end{proposition}
These two facts can be used to prove Corollary \ref{teotresferegrad}:
\begin{proof}[Proof of Corollary \ref{teotresferegrad}.]
We  notice that if $u$ is a solution to \eqref{ellgeneral} and $B(x)=0$  , then $u-u_E$ is a solution as well, where $u_E$ denotes the average of $u$ in $E$. Since the classical Poincar\'e-Wirtinger inequality (see \cite{Meyers}) holds for $u-u_E$, we apply it together with the Caccioppoli inequality and the three spheres inequalities \eqref{tresfere}, \eqref{tresferegrad} to obtain: 
\begin{displaymath} \begin{split}
& \int_{B_{r_2}} \! |\nabla u|^2\le  \frac{C_1}{\rho_0^2(r_3-r_2)^2}\int_{B_{r_2}} \! | u-u_E|^2 \le \\ \le &\frac{C_2}{\rho_0^2(r_3-r_2)^2}  \Big(\int_{B_{r_1}} \! | u- u_E|^2 \Big)^\delta \Big(\int_{B_{r_3}} \! |u- u_E|^2 \Big)^{1-\delta} \le \\ \le &\frac{C_3}{(r_3-r_2)^2}  \Big(\int_{B_{r_1}} \! |\nabla u|^2 \Big)^\delta \Big(\int_{B_{r_3}} \! |\nabla u|^2 \Big)^{1-\delta},
\end{split} \end{displaymath} 
and $C_1$, $C_2$, $C_3$ only depend on $\en$, $n$,  $\frac{r_1}{r_3}$ and $\frac{r_2}{r_3}$.
\end{proof}
We can now show that Proposition \ref{teoPOSC} follows from Proposition \ref{teoPOS}:
\begin{proof}[Proof of Proposition \ref{teoPOSC}.]
From Proposition \ref{teoPOS} we know that 
\begin{displaymath}  \int_{B_{\rho}(x)} \! |\nabla u_i|^2  \ge C_\rho  \int_{\Omega \setminus \overline{D_i}} \! |\nabla u_i|^2 , \end{displaymath}
where $C_\rho$ is given in \eqref{crho}.
We have, using Poincar\`e inequality \eqref{pancarre} and the trace theorem, 
\begin{equation} \label{altofrequenza} \begin{split}  \int_{\Omega\setminus \overline{D_i}}  |\nabla u_i|^2  \ge C \rho_0^{n-2} \norma{u_i}{1}{\Omega \setminus \overline{D_i}}^2  \ge  C \rho_0^{n-2} \norma{g}{\frac{1}{2}}{\partial \Omega}^2. \end{split}\end{equation}
Applying the above estimate to \eqref{POS} and using \eqref{equivalence} will prove our statement.
\end{proof}
Finally, we will need a three spheres inequality for functions that can be written as differences of solutions of the Navier-Stokes equations (due to nonlinearity, this does not follow the previous remarks).
\begin{proposition} \label{tredifferenze} 
Let $u_1$ and $u_2$ be solutions of  
\begin{equation} \label{NSEi}
\left\{ \begin{array}{rl}
 \dive \sigma (u_i,p_i) &= (u_i \cdot \nabla) u_i \hspace{2em} \mathrm{\tmop{in}} \hspace{1em} E,\\ \dive u_i &= 0 \hspace{5.5em} \mathrm{\tmop{in}} \hspace{1em} E,
\end{array} \right.
\end{equation} 
for $i=1,2$. Suppose that $\|u_1 \|_{\mathbf{C}^{1,\alpha}(E)} +\|u_2 \|_{\mathbf{C}^{0,\alpha}(E)} \le \en $. Let $B_R(x) \subset E$. Then there exists a real number $\vartheta^* \in (0, e^{-1/2})$, depending only on $n$, such that, for all $0<r_1 <r_2 <\vartheta^* r_3$ with $r_3 \le R$ we have, calling $w=u_1-u_2$:
\begin{equation} \label{tresferediff}  \int_{B_{r_2}} \! | w|^2 \le C \Big(\int_{B_{r_1} } \! | w|^2  \Big)^\delta  \Big(\int_{B_{r_3}} \! | w|^2  \Big)^{1-\delta} \end{equation} 
where the balls $B_{r_i}$ are centered in $x$, and $\delta \in (0,1)$ and $C>0$ are constants depending only on $\en$, $n$, $\frac{r_1}{r_3}$ and $\frac{r_2}{r_3}$.
\end{proposition}
\begin{proof}
In view of Theorem \eqref{teotresfere} (and the subsequent remarks), it is enough to show that $w$ can be written as a solution to a system of the form  \eqref{ellgeneral}. This is readily done, by subtracting from each other \eqref{NSEi} for $i=1,2$. We may write, in components for $j=1,\dots,n$:
\begin{displaymath}
\mu \triangle w_j  - (u_2)_i \frac{\partial w_j}{\partial x_i} +   \frac{\partial{(u_1)_j}}{\partial x_i} w_i + \frac{\partial (p_1-p_2)}{\partial x_j}=0.
\end{displaymath}  
Calling $(A(x))_{i}= (-u_2)_i $ and $(B(x))_{ij}= \frac{\partial (u_1)_j}{\partial x_i} $, we have that \eqref{ablim} holds in $E$, so the hypotheses of Theorem \ref{teotresfere} hold for $w$. 
\end{proof}
\begin{remark}
We observe that, since the identically zero function solves \eqref{NSEi} in $E$, we can also apply the three spheres inequality to each $u_i$ separately (as we already pointed out in  the previous remark). 
\end{remark}
The proof of Proposition \ref{teoPOS} is now a consequence of the work done so far. 
\begin{proof}[Proof of Proposition \ref{teoPOS}.]
The proof is based upon the validity of the three spheres inequality for \ref{NSE}. Since we have just proved Proposition \ref{tredifferenze}, we refer to \cite{MEE}, \cite{MR} and \cite{MRC} for a detailed proof, which applies here with only slight modifications, as it only requires \eqref{tresfere} and \eqref{tresferegrad} and some geometric constructions which only exploit the regularity of $\partial E$ and thus apply unchanged to this situation. 
\end{proof}

\section{Stability of continuation from Cauchy data.}
We will need the following lemma, proved in \cite{ABRV}:
\begin{lemma}[Regularized domains] \label{regularized} 
Let $\Omega$ be a domain satisfying \eqref{apriori1} and \eqref{apriori2}, and let $D_i$, for $i=1,2$, be two connected open subsets of $\Omega$ satisfying \eqref{apriori3}, \eqref{apriori4}. Then there exist a family of regularized domains $D_i^h \subset \Omega$, for  $0 < h < a \rho_0$, with $C^1$ boundary of constants $\til{\rho_0}$, $\til{M_0}$  and such that 
\begin{equation} \label{643} D_i \subset D_i^{h_1} \subset D_i^{h_2} \; \text{ if  } 0<h_1 \le h_2; \end{equation}
\begin{equation} \label{644} \gamma_0 h \le \mathrm{dist}(x, \partial D_i) \le \gamma_1 h  \; \text{ for all   } x \in  \partial D_i^h; \end{equation}
\begin{equation} \label{645} \mathrm{meas}(D_i^h\setminus D_i)\le \gamma_2 M_1 \rho_0^2 h; \end{equation}
\begin{equation} \label{646} \mathrm{meas}_{n-1}(\partial D_i^h)\le \gamma_3 M_1 \rho_0^2; \end{equation}
and for every $x \in \partial D_i^h$ there exists $y \in \partial D_i$ such that 
\begin{equation} \label{647} |y-x|= \mathrm{dist}(x, \partial D_i), \; \; |\nu(x) - \nu(y)|\le \gamma_4 \frac{h^\alpha}{\rho_0^\alpha}; \end{equation}
where by $\nu(x)$ we mean the outer unit normal to $\partial D_i^h$, $\nu(y)$ is the outer unit normal to $D_i$, and the constants $a$, $\gamma_j$, $j=0 \dots 4$ and the ratios 
$\frac{\til{M}_0}{M_0}$, $\frac{\til{\rho}_0}{\rho_0}$ only depend on $M_0$ and $\alpha$.
\end{lemma}
The proof the stability estimate of continuation from the Cauchy data heavily relies upon the upcoming result, which deals with the estimation of the stability of the stationary Navier-Stokes equations with homogeneous Cauchy data, the proof of which, in turn, is based upon an extension argument. We pospone the proof to the next section. Here we will set up the notations and state the theorem. Let us consider a bounded domain $E\subset \mathbb{R}^n$ satisfying hypotheses \eqref{apriori1} and \eqref{apriori2}, and take $\Gamma \subset \partial E$ a connected open portion of the boundary of  class $C^{2, \alpha}$ with constants $\rho_0$, $M_0$. Let $P_0 \in \Gamma$ such that \eqref{apriori2G} holds. By definition, after a suitable change of coordinates we have that $P_0 = 0$ and 
\begin{equation} 
E \cap B_{\rho_0}(0) = \{ (x^\prime, x_n) \in E  \, \text{ s.t.} \, x_n>\varphi(x^\prime)  \} \subset E,
\end{equation}
where $\varphi$ is a $C^{2,\alpha}(B^\prime_{\rho_0}(0))$ function satisfying 
\begin{displaymath}
 \begin{split}
  \varphi(0)&=0, \\ 
|\nabla \varphi (0)|&=0, \\ 
\|\varphi \|_{C^{2,\alpha} (B^\prime_{\rho_0}(0))}& \le M_0 \rho_0.
 \end{split}
\end{displaymath}
 Define
\begin{equation} \begin{split} \label{rho00}
 \rho_{00} & = \frac{\rho_0}{\sqrt{1+M_0^2}}, \\ \Gamma_0 & = \{ (x^\prime, x_n)   \in \Gamma \, \, \mathrm{s.t.} \, \, |x^\prime|\le \rho_{00}, \, \,  x_n = \varphi(x^\prime) \}.
\end{split} \end{equation}
In what follows we shall consider $(u_i, p_i)$, for $i=1,2$, which are solutions to the following problems:
\begin{equation}
  \label{sisdiff} \left\{ \begin{array}{rl}
    \dive \sigma(u_i,p_i) &= (u_i \cdot \nabla) u_i \hspace{2em} \mathrm{\tmop{in}} \hspace{1em}
    E,\\
    \dive u_i & = 0 \hspace{2em} \mathrm{\tmop{in}} \hspace{1em} E,\\
    u_i & = g \hspace{2em} \mathrm{\tmop{on}} \hspace{1em} \Gamma,\\
    \sigma (u_i, p_i) \cdot \nu & = \psi_i \hspace{2em} \mathrm{\tmop{on}}
    \hspace{1em} \Gamma,\\  \end{array} \right.
\end{equation}
where $g$ satisfies \eqref{apriori5POS}-\eqref{apriori7POS}, $\psi_i \in \mathbf{C}^{1, \alpha} (\Gamma)$ and we use the same notations as we did in Section 4 (which are also to be understood in what follows).
Define $w =u_1-u_2$, $q=p_1-p_2$, these will solve a system of the following form: 
\begin{equation}
  \label{sisw} \left\{ \begin{array}{rl}
    \dive \sigma(w,q) &= A \cdot \nabla w + B w \hspace{2em} \mathrm{\tmop{in}} \hspace{1em}
    E,\\
    \dive w & = 0 \hspace{2em} \mathrm{\tmop{in}} \hspace{1em} E,\\
    w & = 0 \hspace{2em} \mathrm{\tmop{on}} \hspace{1em} \Gamma,\\
    \sigma (w, q) \cdot \nu & = \psi_0 \hspace{2em} \mathrm{\tmop{on}}
    \hspace{1em} \Gamma,\\  \end{array} \right.
\end{equation}
where $A$ and $B$ were explicitated in \eqref{ellgencomp} and $\psi_0=\psi_1-\psi_2$.
We have the following estimate for a solution of systems of the form \eqref{sisw}, the proof of which is delayed to the next section:
\begin{theorem} \label{stabilitycauchy}
Let $A$ be a vector of $C^{0, \alpha}(E)$ of constants $\rho_0$, $M_0$, let $B$ a matrix of class $C^{0,\alpha}(E)$ of constants $\rho_0$, $M_0$, satisfying 
\begin{equation}
\| A \|_{C^{0,\alpha}}+\| B \|_{C^{0,\alpha}} = \en,
\end{equation}
 and let $(w,q)$ be a solution of class $\mathbf{C}^{1,\alpha}(E) \times C^0(E)$ to the problem:
\begin{equation}
  \label{NseHomDir} \left\{ \begin{array}{rl}
    \dive \sigma(w,q) &= A \cdot \nabla w + B w  \hspace{2em} \mathrm{\tmop{in}} \hspace{1em}
   E,\\
    \dive w & = 0 \hspace{2em} \mathrm{\tmop{in}} \hspace{1em} E,\\
    w & = 0 \hspace{2em} \mathrm{\tmop{on}} \hspace{1em} \Gamma,\\
    \sigma (w, q) \cdot \nu & = \psi \hspace{2em} \mathrm{\tmop{on}}
    \hspace{1em} \Gamma,\\  \end{array} \right.
\end{equation}
where $\psi \in \mathbf{H}^{-\frac{1}{2}} (\Gamma)$ and $\Gamma \subset \partial E$ is of class $C^{1,\alpha}$. 
Then there exists $\widehat{\rho}$, only depending on $M_0$, $\alpha$ and $\en$, such that, letting $P^* = P_0 + \frac{\widehat{\rho}}{4} \nu$ where $\nu$ is the outer normal field to $\partial E$, we have: 
\begin{equation} \label{NseHomDirEqn}
\| w \|_{{\bf L}^\infty(E \cap B_{\frac{3 \widehat{rho}}{8}} (P^*))} \leq \frac{C}{\rho_0^{\frac{n}{2}}} \normadue{w}{E}^{1-\tau} (\rho_0 \|\psi  \|_{{\bf H}^{-\frac{1}{2}} (\Gamma)})^\tau,
\end{equation}
where $\tau$ only depends on $\alpha$ and $M_0$ and $C>0$ also depends on $\en$. 
\end{theorem}
\begin{proof}[Proof of Proposition \ref{teostabest}]
Let $\theta= \mathrm{min} \{a, \frac{7}{8 \gamma_1} \frac{1}{2\gamma_0 (1+M_0^2)} \}$ where $a$, $\gamma_0$, $\gamma_1$ are the constants depending only on $M_0$ and $\alpha$ introduced in Lemma \ref{regularized}, then let $\overline{\rho}= \theta \rho_0$ and fix $\rho \le \overline{\rho}$.
Let $D_1^\rho$, $D_2^\rho$ be the regularized domain associated with $D_1$, $D_2$ respectively, built according to Lemma \ref{regularized}. Let $G$ be the connected component of $\Omega\setminus(\overline{D_1 \cup D_2})$ containing $\partial \Omega$, and $G^\rho$ be the connected component of $\overline{\Omega}\setminus(D_1^\rho \cup D_2^\rho)$ which contains $\partial \Omega$.
We have  that \begin{equation*}
D_2 \setminus \overline{D_1} \subset \Omega_1 \setminus \overline{G} \subset \big( (D_1^\rho \setminus \overline{D_1} ) \setminus\overline{G}\big) \cup \big( (\Omega \setminus G^\rho)\setminus D_1^\rho \big)
\end{equation*} 
and 
\begin{equation} \label{decomposizione}
\partial \big( (\Omega \setminus G^\rho)\setminus D_1^\rho \big) = \Gamma_1^\rho \cup \Gamma_2^\rho,
\end{equation}
where $\Gamma_2^\rho= \partial D_2^\rho \cap \partial G^\rho$ and $\Gamma_1^\rho \subset \partial D_1^\rho$. Then 
\begin{equation} \label{652} \int_{D_2 \setminus \overline{D_1 }} |\nabla u_1|^2 \le \int_{\Omega_1 \setminus \overline{G}} |\nabla u_1|^2 \le \int_{(D_1^\rho \setminus \overline{D_1} )\setminus\overline{G}} |\nabla u_1|^2 +\int_{(\Omega \setminus G^\rho)\setminus D_1^\rho} |\nabla u_1|^2. \end{equation} 
The first term can be estimated directly, using \eqref{schauder1} and \eqref{645} we have 
\begin{equation} \label{6.53} \int_{(D_1^\rho \setminus \overline{D_1} )\setminus\overline{G}} |\nabla u_1|^2 \le C \rho_0^{n-2} \gio^2 \frac{\rho}{\rho_0} \end{equation}
where $C$ only depends on the $M_0$, $M_1$, $\alpha$ and $\mu$.
We call $\Omega(\rho)= (\Omega \setminus G^\rho)\setminus D_1^\rho$. We use the first equation in \eqref{NSE}, multiply it by $u_1$ and integrate over $\Omega(\rho)$ to derive the following identity:
\begin{equation*}
\begin{split} 
 0 &= \int_{\Omega(\rho)} \mu u_1 \cdot  \triangle u_1 - u_1 (u_1 \cdot \nabla) u_1 - u_1 \cdot \nabla p_1 = \\ 
   &= \int_{\Omega(\rho)} \mu u_1 \cdot  \triangle u_1 - \frac{1}{2} \dive (u_1 |u_1|^2) - u_1 \cdot \nabla p_1   = \\ 
   &= \int_{\partial\Omega(\rho)} \mu (\nabla u_1 \cdot \nu) u_1  - \frac{1}{2} u_1 \cdot \nu |u_1|^2 - (u_1 \cdot \nu) p - \mu \int_{\Omega(\rho)} |\nabla u_1|^2, 
\end{split} 
\end{equation*} 
which yields, recalling \eqref{decomposizione}, 
\begin{equation} \label{sommandi} 
\begin{split}
 \mu \int_{\Omega(\rho)} |\nabla u_1|^2 =&   \int_{\Gamma_1^\rho}  \mu (\nabla u_1 \cdot \nu) u_1  - \frac{1}{2} u_1 \cdot \nu |u_1|^2 - (u_1 \cdot \nu) p\, + \\+&  \int_{\Gamma_2^\rho} \mu (\nabla u_1 \cdot \nu) u_1  - \frac{1}{2} u_1 \cdot \nu |u_1|^2 - (u_1 \cdot \nu) p  . 
\end{split}
\end{equation}
We start by estimating the first integral on the right hand side of \eqref{sommandi}. If $x \in \Gamma_1^\rho$, by Theorem \ref{regularized}, we find $y \in \partial D_1$ such that $|y-x|= d(x, \partial D_1) \le \gamma_1 \rho$; since $u_1(y)=0$, by Lemma \ref{teoschauder} we have 
\begin{equation} \label{pezzobuono}  |u_1(x)|= |u_1(x)-u_1(y)|\le  C \frac{\rho}{\rho_0}  \gio . \end{equation}
On the other hand, if $x \in \Gamma_2^\rho$, there exists $y \in D_2$ such that $|y-x| = d(x, \partial D_2) \le \gamma_1 \rho$. Again, since $u_2(y)=0$, we have 
\begin{equation} \label{pezzocattivo} \begin{split} & |u_1(x)| \le  |u_1(x)-u_1(y)|+|u_1(y)-u_2(y) |   \\ 
& \le C  \big( \frac{\rho}{\rho_0} \gio + \max_{\partial G^\rho \setminus \partial \Omega} |w| \big) , \end{split}\end{equation} 
where $w=u_1-u_2$.  Combining \eqref{pezzobuono}, \eqref{pezzocattivo} and \eqref{sommandi} and recalling \eqref{schauder1} and \eqref{646} we have:
\begin{equation} \label{sommandi2}
\int_{D_2\setminus D_1} |\nabla u_1|^2  \le \frac{C}{\mu}\rho_0^{n-2}( \gio +1+\mu) \Big( \gio^2 \frac{\rho}{\rho_0} + \gio  \max_{\partial G^\rho \setminus \partial \Omega} |w|  \Big)
\end{equation}
We now need to estimate $\max_{\partial G^\rho \setminus \partial \Omega} |w| $.  We will do so by means of Proposition \ref{tredifferenze}. Take  $ x \in \partial G^\rho \setminus \partial \Omega$ and define \begin{equation} \label{rhostar} \rho^*=\min \left\{ \frac{\rho_0}{16(1+M_0^2)}, \frac{\widehat{\rho}}{16}   \right\}, \end{equation}
\begin{equation}\label{zetazero}
x_0= P_0 - \frac{\min \{\rho_{00},\widehat{\rho}  \}}{16}\nu,
\end{equation}
where $\nu$ is the outer normal to $\partial \Omega$ at the point $P_0$. By construction $x_0 \in \overline{\til{\Omega}_{\frac{\rho^*}{2}}}$. 
There exists an arc $\gamma:  [0,1]  \mapsto G^\rho \cap \overline{\til{\Omega}_{\frac{\rho^*}{2}}} $ such that $\gamma(0)=x_0$, $\gamma(1)=x$ and $\gamma([0,1])\subset G^\rho \cap \overline{\til{\Omega}_{\frac{\rho^*}{2}}}$. Let us define \begin{equation}
\rho_3= \min \{ \gamma_0 \rho \vartheta^*, \widehat{\rho} \}, \; \rho_2= \frac{3}{4} \rho_3, \; \rho_1=\frac{1}{4} \rho_3,       \end{equation}
where $\vartheta^*$ is the constant given in Theorem \ref{teotresfere}.
We pick a sequence of $S+1$  times $t_i$ and points $x_i= \gamma(t_i)$,  $i=0 \dots S$, defined by the following construction. Call $t_0=0$, then:
\begin{displaymath}  \begin{split}
t_i&= \mathrm{max} \big\{ t\in (0,1]  \text{ s.t. } |\gamma(t)- x_i| = \frac{\rho_3}{2} \big\}, \text{ if } |x_i-x| > \frac{\rho_3}{2}
, \\
&\text{otherwise,   } i=S,
\end{split}\end{displaymath}
and stop the process.  The number of spheres is bounded by 
\begin{displaymath} S\le C \Big( \frac{\rho_0}{\rho} \Big)^n \end{displaymath} where $C$ only depends on $\alpha$, $M_0$, $M_1$ and $\en$.
The balls $B_{\frac{\rho_3}{2}}(x_i)$ are pairwise disjoint, the distance between two consecutive centers is given by 
\begin{displaymath}| x_{i+1}-x_i | =\frac{\rho_3}{2}, \, \; i=0 \dots S-1, \, \;  |x_S - x| \le \frac{\rho_3 }{2}. \end{displaymath}  
We iterate the three spheres inequality \eqref{tresferediff} on a chain of spheres with radii $\rho_1$, $\rho_2$ and $\rho_3$, this leads us to 
\begin{equation} \label{iteratresfere}  \int_{B_{\rho_2} (x)} \! | w|^2 dx \le C \Big(\int_G  \! | w|^2 dx \Big)^{1-\delta^S}  \Big(\int_{B_{\rho_3}(x_0)} \! | w |^2 dx \Big)^{\delta^S} \end{equation}
where $0<\delta<1$ and $C>0$ only depend on $M_0$, $\alpha$ and $\en$. From our choice of $\bar{\rho}$ and $\vartheta^*$, it follows that $B_{\rho_1}(x_0) \subset B_{\rho^*}(x_0) \subset G \cap B_{\frac{3 \rho_1 }{4}}(P^*)$, where we follow the notations from Theorem \ref{stabilitycauchy}. 
We may apply Theorem \ref{stabilitycauchy} to $w$; thus, using \eqref{NseHomDirEqn}, \eqref{stimanormadiretto}, \eqref{HpPiccolo} and \eqref{aprioriE} on \eqref{iteratresfere} we have:
\begin{equation} \label{pallina}
 \int_{B_{\rho_2}(x)} \! | w|^2 dx \le C \rho_0^{n-2} \epsilon^{2 \tau \delta^S}.
\end{equation}
The following interpolation inequality holds for all functions $v$ defined on the ball $B_t(x) \subset \mathbb{R}^n$:
\begin{equation} \label{interpolation}
\|v \|_{\mathbf{L}^\infty (B_t(x))} \le C \Big( \Big(  \int_{B_t(x)} | v|^2 \Big)^{\frac{1}{n+2}} |\nabla v|^{\frac{n}{n+2}}_{\mathbf{L}^\infty (B_t(x))} +   \frac{1}{t^{n/2}} \Big( \int_{B_t(x)} | v|^2 \Big)^{\frac{1}{2}} \Big)
\end{equation}
We apply it to $w$ in $B_{\rho_2}(x)$, using \eqref{pallina} and \eqref{schauder1} we obtain
\begin{equation} \label{stimaw}
\| w \|_{\mathbf{L}^\infty (B_{\rho_2}(x))} \le C \Big( \frac{\rho_0}{\rho} \Big)^{\frac{n}{2}}  \epsilon^{\gamma  \delta^S},
\end{equation}
where $\gamma=\frac{2\tau}{n+2}$. 
Finally, from \eqref{stimaw} and \eqref{sommandi2}, and recalling \eqref{aprioriE} we get:
\begin{equation} \label{sommandi3} 
\int_{D_2\setminus D_1} |\nabla u_1|^2  \le C \rho_0^{n-2}  \Big( \frac{\rho}{\rho_0} + \Big( \frac{\rho_0}{\rho} \Big)^{\frac{n}{2}} \epsilon^{\gamma \delta^S}  \Big),
\end{equation}
with $C$ only depending on $M_0$, $\alpha$, $\mu$ and $\en$.
Let us now choose $\rho$ depending upon $\epsilon$, of the form 
\begin{displaymath}
 \rho(\epsilon) = \rho_0 \Bigg( \frac{2S \log |\delta|}{\log |\log \epsilon^\gamma|} \Bigg)^{-\frac{1}{n}}.
\end{displaymath}
We have that $\rho$ is defined and increasing in the interval $(0, e^{-1})$. Call $\overline{\zeta}$ the number such that $\rho(\overline{\zeta}) = \min \{ \overline{\rho}, \widehat{\rho} \}$, and let $\til{\zeta}= \min \{\overline{\zeta}, \exp (-\gamma^2) \}$. Since the thesis is trivial for larger values of $\epsilon$, it is not restrictive to prove it only in the smaller interval $(0, \tilde{\zeta})$. We are able to apply \eqref{sommandi3} to \eqref{652} with $\rho=\rho(\epsilon)$ for $ \epsilon \in (0, \tilde{\zeta})$ to obtain 
\begin{equation}
\label{quasifinito}
\int_{D_2 \setminus D_1} |\nabla u_1|^2 \le  C \rho_0^{n-2}  \log |\log \epsilon|^\gamma,
\end{equation}
and since $\epsilon \le \exp(-\gamma^2)$  it is elementary to prove that \begin{displaymath}
\log |\log \epsilon^\gamma| \ge \frac{1}{2} \log | \log \epsilon|,
\end{displaymath}
so that \eqref{quasifinito} finally reads 
\begin{displaymath}
\int_{D_2 \setminus D_1} |\nabla u_1|^2 \le  C \rho_0^{n-2} \,\omega(\epsilon),
\end{displaymath}
with  $\omega(t) = \log |\log t|^{\frac{1}{n}}$ defined for all $0<t<e^{-1}$, and $C$ depends on $M_0$, $M_1$, $\alpha$, $\mu$ and $\en$.
\end{proof}
\begin{proof}[Proof of Proposition \ref{teostabestimpr}]
We will prove the thesis for $u_1$, the case $u_2$ being completely analogous. 
First of all, we observe that
\begin{equation} 
\label{sommandiB}
 \int_{D_2 \setminus D_1} \mu |\nabla u_1|^2  \le  \int_{\Omega_1 \setminus G} \mu |\nabla u_1|^2   = \int_{\partial (\Omega_1 \setminus G)} \mu (\nabla u_1 \cdot \nu) u_1 -  p_1 ( u_1 \cdot \nu)  - \frac{1}{2} u_1\cdot \nu |u_1|^2,
\end{equation}
and that 
\begin{equation*}
\partial (\Omega_1 \setminus G) \subset \partial D_1 \cup (\partial D_2 \cap \partial G).
\end{equation*}
If we recall the no-slip condition, apply to \eqref{sommandiB} computations similar to those in \eqref{652}, \eqref{6.53}, we get
\begin{equation} \label{sommandiagain}
\int_{D_2 \setminus D_1} |\nabla u_1|^2  \le \frac{C}{\mu}\rho_0^{n-2} (\gio + 1+ \mu)     \max_{\partial D_2 \cap \partial G} |w|,
\end{equation}
where again $w= u_1 - u_2$, and $C$ only depends on $\alpha$, $M_0$ and $M_1$. Take a point $z \in \partial G$. To evaluate  $\max |w|$ on ${\partial D_2 \cap \partial G}$, we start by choosing a point $z \in \partial G$ and estimating  $\| \nabla u \|$ on a ball centered in $z$, in terms of $\| \nabla u \|$ evaluated on a ball centered in $x_0$. We will do so by applying iteratevely the three spheres inequality, twice.   By exploiting the regularity assumptions on $\partial G$, we find a 
cone centered in $z$, which we denote by $C(z, \xi, \vartheta_0)$ (where $\vartheta_0 =\arctan \frac{\rho_0}{M_0}$ is  half the aperture of the cone and $\xi \in \mathbb{R}^{n}$ is a unit vector representing the direction of the cone),  such that 
$C(z, \xi, \vartheta_0) \cap B_{\tilde{\rho}_0} (z) \subset G$. It can be shown (\cite[Proposition 5.5]{ARRV}) that $G_\rho$ is connected for $\rho \le \frac{\tilde{\rho}_0 h_0 }{3}$ with $h_0$ only depending on $M_0$. 
We now claim (without proof, see \cite{MEE} and \cite{MRC} for the detailed constructions in the same context) that we may build $\lambda_1 >0$ and $\theta_1 >0$, such that, if we define
\begin{equation*}\begin{split}
w_1 &=z+ \lambda_1 \xi, \\ 
\rho_1 &= \vartheta^* h_0 \lambda_1 \sin \vartheta_1. 
\end{split}\end{equation*}
where $0<\vartheta^*\le 1$ was introduced in Theorem \ref{teotresfere}, then the following claims hold: $B_{\rho_1}(w_1) \subset C(z, \xi, \vartheta_1) \cap B_{\tilde{\rho}_0}(z)$ and $B_{\frac{4 \rho_1}{\vartheta^*}}(w_1) \subset C(z, \xi, \vartheta_0) \cap B_{\tilde{\rho}_0}(z) \subset G$. Furthermore $\frac{4 \rho_1}{\vartheta^*} \le  \rho^*$, hence $B_{\frac{4 \rho_1}{\vartheta^*}} (x_0) \subset G$, where $\rho^*$ and $x_0$ were defined by \eqref{rhostar} and \eqref{zetazero} respectively. It follows that $w_1$, $x_0 \in \overline{G_{\frac{4\rho_1}{\vartheta^*}}}$, which is connected by construction.
Iterating the three spheres inequality \eqref{tresferediff} (mimicking the construction made in the previous proof) 
\begin{equation} \label{iteratresferei}  \int_{B_{\rho_1} (w_1)} \! | w|^2 dx \le C \Big(\int_G  \! | w|^2 dx \Big)^{1-\delta^S}  \Big(\int_{B_{\rho_1 }(x_0)} \! | w |^2 dx \Big)^{\delta^S} \end{equation}
where $0<\delta<1$ and $C \ge 1$ depend only on $n$ and $\en$, and $S \le \frac{M_1 \rho_0^n}{\omega_n \rho_1^n}$. 
We apply Theorem \ref{stabilitycauchy} in the same fashion as the previous proof, which leads to 
\begin{equation*}
\int_{B_{\rho_1}(w_1)} |w|^2 \le C \rho_0^n  \epsilon^{2\beta},
\end{equation*}
where $0<\beta<1$ and $C >0$ only depend on $\alpha$, $M_0$, $\en$ and $\frac{\tilde{\rho}_0}{\rho_0}$. 
So far the estimate we have is only on a ball centered in $w_1$, we need to approach $z \in \partial G$ using a sequence of balls, all contained in $C(z, \xi, \vartheta_1)$, by suitably shrinking their radii. Take 
\begin{equation*}
\chi = \frac{1-\sin \vartheta_1}{1+\sin\vartheta_1}
\end{equation*}
and define, for $k \ge 2$,
\begin{equation*} \begin{split}
\lambda_k&=\chi \lambda_{k-1}, \\
\rho_k&= \chi \rho_{k-1}, \\
w_k &= z + \lambda_k \xi. \\
\end{split}\end{equation*}
With these choices, $\lambda_k= \lambda \chi^{k-1} \lambda_1$, $\rho_k=\chi^{k-1} \rho_1$ and $B_{\rho_{k+1}}(w_{k+1}) \subset B_{3\rho_k}(w_k)$, $B_{\frac{4}{\vartheta^*}\rho_k}(w_k) \subset C(z, \xi, \vartheta_0) \cap B_{\tilde{\rho}_0}(z) \subset G$.
Denote by
\begin{displaymath}
d(k)= |w_k-z|-\rho_k,
\end{displaymath}
we also have
\begin{displaymath}
d(k)= \chi^{k-1}d(1),
\end{displaymath}
with
\begin{displaymath}
d(1)= \lambda_1(1-\vartheta^* \sin \vartheta_1).
\end{displaymath}
Now take any $\rho \le d(1)$ and let $k=k(\rho)$ the smallest integer such that $d(k) \le \rho$, explicitly
\begin{equation} \label{chirho}
\frac{\big|\log \frac{\rho}{d(1)}\big|}{\log \chi} \le k(\rho)-1 \le \frac{| \log \frac{\rho} {d(1)}|}{\log \chi}+1.
\end{equation}
We iterate the three spheres inequality \eqref{tredifferenze} over the chain of balls centered in $w_j$ and radii $\rho_j$, $3 \rho_j$, $4\rho_j$, for $j=1, \dots, k(\rho)-1$, which yields
\begin{equation} \label{iteratresferetre}
\int_{B_{\rho_{k(\rho)}}(w_{k(\rho)})} |w|^2 \le C  \rho^n \epsilon^{2 \beta \delta^{k(\rho)-1}},
\end{equation}
with $C$ only depending on $\alpha$, $M_0$, $\en$ and $\frac{\tilde{\rho}_0}{\rho_0}$.
Using the interpolation inequality \eqref{interpolation} and \eqref{schauder1} we obtain 
\begin{equation}\label{543}
\|w \|_{\mathbf{L}^\infty (B_{\rho_{k(\rho)}}(w_{k(\rho)}))} \le C  \frac{\epsilon^{\beta_1 \delta^{k(\rho)-1}}}{\chi^{\frac{n}{2}(k(\rho)-1)}},
\end{equation}
where $\beta_1=\frac{2 \beta}{n+2}$ depends only on $\alpha$, $M_0$, $M_1$, $\en$ and $\frac{\tilde{\rho}_0}{\rho_0}$.
From \eqref{543}) and \eqref{schauder1} we obtain 
\begin{equation} \label{544}
|w(z) | \le C \Bigg( \frac{\rho}{\rho_0} +\frac{\epsilon^{\beta_1 \delta^{k(\rho)-1}}}{\chi^{\frac{n}{2}(k(\rho)-1)}} \Bigg),
\end{equation}
Finally, call
\begin{displaymath}
\rho(\epsilon)= d(1) |\log \epsilon^{\beta_1}|^{-B},
\end{displaymath}
with
\begin{displaymath}
B= \frac{|\log \chi|}{2 \log |\delta|}.
\end{displaymath}
and let $\tilde{\zeta} = \exp(-\beta_1^{-1})$. We have that $\rho(\epsilon)$ is monotone increasing in the interval $0<\epsilon < \tilde{\zeta}$, and $\rho(\tilde{\zeta})=d(1)$, so $\rho(\epsilon) \le d(1)$ there.  By choosing $\rho=\rho(\epsilon)$ from \eqref{544} and \eqref{sommandiagain} we obtain
\begin{equation}
\int_{D_2 \setminus D_1} |\nabla u_1|^2 \le C \rho_0^{n-2}
|\log \epsilon|^{-B},
\end{equation}
where $C$ only depends on $\mu$, $\alpha$, $M_0$, $\en$ and $\frac{\til{\rho}_0}{\rho_0}$.
\end{proof}

\section{Proof of Theorem \ref{stabilitycauchy}. }
This section is entirely devoted to the proof of Theorem \ref{stabilitycauchy}. Let us start by introducing some notations.
We define  $$Q(P_0) = B^\prime_{\rho_{00}} (0) \times \Big[-\frac{M_0\rho_0^2}{\sqrt{1+M_0^2}}, \frac{M_0\rho_0^2}{\sqrt{1+M_0^2}}\Big],$$  
and 
\begin{equation}\label{gagrafico} \begin{split}
\Gamma_0 &= \partial E \cap Q(P_0). \\ 
\end{split}\end{equation}
Finally, let us call $E^- = Q(P_0) \setminus  E$ and $\til{E} = E \cup E^- \cup \Gamma_0$.
We start the proof by choosing a vector $\til{A}$ and a matrix $\til{B}$ such that $A=\til{A}$, $B=\til{B}$ in $E$, and 
\begin{equation} \label{ablim3}
\| \til{A} \|_{C^{1,\alpha}(\til{E})} + \| \til{B} \|_{C^{0,\alpha}(\til{E})}  \le C_1 \en,
\end{equation}
where $C_1 >0$ only depends on $\alpha$ and $M_0$. Our aim is to build an extension $\til{w}$ of $w$ such that it satisfies the extended problem
\begin{equation} 
\begin{split} 
\label{sistilde}
\dive \sigma (\til{w}, \til{q}) &=\til{A} \cdot \nabla \til{w}+\til{B} w+ \Phi \, \, \text{  in  } \, \, \til{E}, 
\\  \dive \til{u}&=0  \, \, \text{  in  } \, \, \til{E},
\end{split}
\end{equation}
where $\Phi \in \accan{-1}{\til{E}}$ is such that
\begin{equation} \label{stimaPhi}
\norma{\Phi}{-1}{\til{E}} \le C \| \psi \|_{\mathbf{H}^{-\frac{1}{2}}(\Gamma) }.
\end{equation}

Define $\til{w}$, $\til{q}$ as follows:
\begin{displaymath}\til{w}= \left\{ \begin{array}{rl} &  w \quad \text{ in } \; E, \\ & 0 \quad \text{ in } \; E^-. \end{array} \right. \end{displaymath}
\begin{displaymath}\til{q}= \left\{ \begin{array}{rl} &  q \quad \text{ in } \; E, \\ & 0 \quad \text{ in } \; E^-. \end{array} \right. \end{displaymath}
We have that $\til{w} \in \mathbf{H}^1 (\til{E})$ and that $\dive \til{w} =0$, in the weak sense in $\til{E}$.


In order to write a system \eqref{sistilde} for $(\til{w},\til{q})$, we take any $v \in \accano{1}{\til{E}}$ and consider
\begin{equation}
  \label{NSEEXT} 
\int_{\til{E}} \sigma (\til{w}, \til{q}) \cdot \nabla v 
= \int_{E} \sigma ( w, q) \cdot \nabla v + \int_{E^-} \sigma( \til{w}, \til{q} ) \cdot \nabla v.
\end{equation}
By the divergence theorem on the first term we obtain 
\begin{equation}
  \label{phiuno}  \int_{E} \sigma(w,q) \cdot \nabla v= -\int_{E} (A \cdot \nabla w) \cdot \nabla v - \int_{E}   B w \cdot \nabla v  + \int_{\Gamma} \psi \cdot v. \end{equation}
Define $ \Phi(v)= \int_{\Gamma} \psi \cdot v $ for all  $v \in \accano{1}{\til{E}}$.
Using \eqref{phiuno} and the trace theorem:
\begin{equation} \label{stimaphi1} \big| \Phi(v) \big| \le \norma{\psi}{-\frac{1}{2}}{\Gamma} \norma{v}{\frac{1}{2}}{\Gamma} \le C  \rho_0\psii \norma{v}{1}{E^-}, \end{equation}

therefore $\til{w}$ satisfies \eqref{sistilde} in the weak sense, with $\Phi$ such that \eqref{stimaPhi} holds. 
We now want to apply the three spheres inequality to the inhomogeneous system \eqref{sistilde}. In order to do so, we need to establish local well posedness for the Cauchy problem for the linearized Navier Stokes equations. 
We claim the following: there exists $\widehat{\rho}$ such that the problem
\begin{equation}
 \label{NSEPARTIC}
\left\{ \begin{array}{rl}
    \dive \sigma (w^*, q^*) &= \til{A} \cdot \nabla w^* + \til{B} w^* +\Phi \hspace{2em} \mathrm{\tmop{in}} \hspace{1em} B_{\widehat{\rho}}, \\ 
\dive w^* &=0 \hspace{2em} \mathrm{\tmop{in}} \hspace{1em} B_{\widehat{\rho}},\\
    w^* & = 0 \hspace{2em} \mathrm{\tmop{on}} \hspace{1em} \partial B_{\widehat{\rho}}.
  \end{array} \right.
\end{equation}
admits a weak solution $w^*$ in the ball $B_{\widehat{rho}}$ such that 
\begin{equation}
 \label{sispallastima}
\| w^*\|_{\mathbf{H}^1_0 (B_{\widehat{\rho}})} \le C \rho_0  \| \Phi\|_{\mathbf{H}^{-1} (E^-)},
\end{equation}
This can be shown by projecting \eqref{NSEPARTIC} on the space of divergence free function, so that it becomes a pressure free second order partial differential equation in $w^*$ only, with the principal part being the laplacian operator, and the lower order terms are continuous and bounded by a constant depending only on $\alpha$, $M_0$ and $\en$. Therefore the existence of a solution and the well posedness of \eqref{NSEPARTIC} is shown by a standard coercivity argument for sufficiently small radii.  Using linearity we write $\til{w}= w_0+w^*$,  with $w^* \in \accano{1}{B_{\widehat{\rho}}}$ such that $(w^*,q^*)$ solves \eqref{NSEPARTIC} and $w_0$ solves
\begin{equation}  \label{NSEHOM} 
\left\{ \begin{array}{rl}
    \dive \sigma (w_0, q_0) + \til{A} \cdot \nabla w_0 + \til{B} w_0 &= 0 \hspace{2em} \mathrm{\tmop{in}} \hspace{1em} B_{\rho_1}, \\ 
    \dive w_0 & = 0 \hspace{2em} \mathrm{\tmop{on}} \hspace{1em} B_{\rho_1}.\\
  \end{array} \right.
  \end{equation}
Using interior regularity of solutions for elliptic systems we get
\begin{equation} 
 \| w_0 \|_{{\bf L}^\infty( B_{\frac{t}{2}} (x))} \le t^{-\frac{n}{2}} \normadue{w_0}{B_{\frac{t}{2}}(x)}.
\end{equation}
We will thus need to estimate $\normadue{w_0}{B(x)}$ on a  ball near the boundary. 
We  may apply Theorem \ref{teotresfere} to $w_0$, thus, calling   and $r_3= \widehat{\rho}$,  $r_1= \frac{r_3}{8}$, $r_2= \frac{3  \, r_3}{8}$ we have (understanding that all balls are centered in $P_0$) 
\begin{equation} \label{3sfereu0} 
\normadue{w_0}{B_{r_2}} \le C \normadue{w_0}{B_{r_1}}^{\tau} \normadue{w_0}{B_{r_3}}^{1-\tau},
\end{equation}
with $C>0$ only depending on $n$ and $\en$. Let us call $\eta=\rho_0\norma{\psi}{-\frac{1}{2}}{\Gamma}$.
By the triangle inequality we have that
\begin{equation} \label{trin1}
\normadue{w_0}{B_{r}} \le \normadue{\til{w}}{B_{r}}+\normadue{w^*}{B_{r}} \le \normadue{\til{w}}{B_{r}}  + C \eta, 
\end{equation}
for $r=r_1,r_3$; furthermore, 
\begin{equation} \label{trin2}
\normadue{\til{w}}{B_{r_2}} \le \normadue{w_0}{B_{r_2}}+\normadue{w^*}{B_{r_2}} \le \normadue{w_0}{B_{r_2}}  + C \eta, 
\end{equation}
where (in both cases) $C>0$ only depends on $n$ and $\en$.
Putting together \eqref{3sfereu0}, \eqref{trin1}, \eqref{trin2} we get
\begin{equation} \begin{split} \label{3sfere2}
& \normadue{w}{B_{r_2}} \le \normadue{\til{w}}{B_{r_2} \cap E} \le \\  \le & C \eta + C (\normadue{\til{w}}{B_{r_1}}+ C \eta)^{\tau}  (\normadue{\til{w}}{B_{r_3} \cap E} + C \eta )^{1-\tau} \le \\ \le & C \big( \eta + \eta^\tau (\eta + \normadue{w}{E} )^{1-\tau} \big) \le C \eta^\tau \normadue{w}{E}^{1-\tau}.  \end{split} \end{equation}

\end{document}